\documentclass[12pt]{article}

\usepackage{latexsym}
\usepackage{amssymb}
\usepackage{amsmath}
\usepackage{enumerate}
\usepackage{amsmath, amsthm,amsfonts,amssymb,latexsym}
\usepackage[all]{xy}
\usepackage{enumerate}

\newcommand{\eps}{\varepsilon}

\newcommand{\h}{\mathcal{H}}

\newcommand{\Amc}{\mathcal{A}}
\newcommand{\Bmc}{\mathcal{B}}
\newcommand{\Cmc}{\mathcal{C}}
\newcommand{\Dmc}{\mathcal{D}}
\newcommand{\A}{\mathcal{A}}
\newcommand{\B}{\mathcal{B}}
\newcommand{\C}{\mathcal{C}}
\newcommand{\F}{\mathcal{F}}
\newcommand{\Fg}{\mathbb{F}}
\newcommand{\J}{\mathcal{J}}

\newcommand{\W}{\mathcal{W}}

\newcommand{\BOP}{\mathbf{B}}
\newcommand{\LH}{\mathbf{L}(\mathcal{H})}
\newcommand{\BH}{\mathbf{B}(\mathcal{H})}

\newcommand{\Complex}{\mathbb{C}}

\newcommand{\abs}[1]{\left\vert#1\right\vert}
\newcommand{\set}[1]{\left\{#1\right\}}
\newcommand{\seq}[1]{\left<#1\right>}
\newcommand{\norm}[1]{\left\Vert#1\right\Vert}

\newcommand{\lz}{l^1(\mathbb{Z})}
\DeclareMathOperator{\diag}{diag} 
\DeclareMathOperator{\D}{D^*} \DeclareMathOperator{\rad}{Rad}
\newcommand{\algebra}[1]{\Complex\langle #1 \rangle}

\newtheorem{theorem}{Theorem}
\newtheorem*{thmp}{Theorem (V. Paulsen)}
\newtheorem{lemma}[theorem]{Lemma}
\newtheorem{proposition}[theorem]{Proposition}
\newtheorem{corollary}[theorem]{Corollary}
\newtheorem{definition}[theorem]{Definition}
\newtheorem{example}[theorem]{Example}
\newtheorem{remark}[theorem]{Remark}

\begin{document}
\title{$*$-Doubles and embedding of associative algebras in $\BH$. }
\author{Stanislav Popovych}
\date{}
\maketitle

 \footnotetext{ 
(Secondary) }

\begin{abstract}

We prove that an associative algebra $\A$ is isomorphic to a
subalgebra of a  $C\sp*$-algebra  if and only if its $*$-double $\A
 * \A^*$ is $*$-isomorphic to a $*$-subalgebra of a
$C\sp*$-algebra. In particular each operator algebra is shown to be
completely boundedly isomorphic to an operator algebra $\B$ with the
greatest $C\sp*$-subalgebra consisting  of the  multiples of the
unit and such that each element in $\B$ is determined by its module
up to a scalar multiple. We also study the maximal subalgebras of an
operator algebra $\A$ which are mapped into $C\sp*$-algebras under
completely bounded faithful representations of $\A$.

\end{abstract}

\section{Introduction}

There  are  well-know characterizations   of Banach algebras
bi-continuously isomorphic to  closed subalgebras in the algebra
$\BH$ of bounded linear operators on Hilbert space $\h$ due to
Varopoulos~\cite{Varopoulos} and  P. G. Dixon~\cite{Dixon}. In
sequel subalgebras of $\BH$ will be called {\it operator algebras}.

With an advent of operator space theory a useful characterization of
operator algebras was given by Blecher, Ruan ans
Sinclair~\cite{BRS}. We will collect  some necessary definitions and
facts from the theory below.
 Let $\mathcal{B}$ be a unital operator algebra in $\BH$.
 The algebra $M_{n}(\BH)$ of $n\times n$ matrices with entries in
$\BH$ has a norm $\|\cdot\|_{n}$ via the identification of
$M_n(\BH)$ with $B(\h^n)$, where $\h^n$ is the direct sum of $n$
copies of a Hilbert space $\h$. The algebra  $M_n(\B)$ inherits a
norm $\|\cdot\|_n$ via natural inclusion into $M_n(\BH)$. The  norms
$ \|\cdot\|_n $  are called matrix norms on the operator algebra
$\B$. If $\phi\colon \B \to \B_1$ is a linear bounded map between
two operator algebras then $\phi^{(n)} = id\otimes \phi$ maps
$M_n(\B)$ into $M_n(\B_1)$ and $\norm{\phi}_{cb}=\sup\limits_n
\norm{\phi^{(n)}}$ is called {\it completely bounded norm} of
$\phi$.
 The map $\phi$ is called {\it completely bounded}  if $\norm{\phi}_{cb}<\infty$.

An abstract operator algebra  $\A$ is an associative algebra with a given collection of norms
$\norm{\cdot}_{m,n}$ on $M_{m,n}(\mathcal{A})$  satisfying certain axioms (see~\cite{BRS}).
 The BRS Theorem states that every unital abstract operator algebra
 with the unit
of norm 1 is completely isometric to an operator algebra in $\BH$.

For non-normed associative algebras no characterization seems to be
known. Not every associative algebra is isomorphic to a subalgebra
of a Banach algebra. To see this consider the quotient $\A$ of the
free algebra on two generators $x$ and $y$ by the ideal generated by
$x y -yx -x$. Then for every $n\ge 1$ we will have the relation $x^n
y - y x^n = n x^n$. Thus for every algebra norm on $\A$ we will have
$n \norm{x^n} \le 2 \norm{x^n}\norm{y}$. Hence $x^n$ should be zero
for some $n$ which  can be easily checked to be false.  This
contradiction shows that $\A$ can not be embedded into a Banach
algebra. The algebra  $\BOP(\lz)$ gives an example of an algebra
which admits a norm but which is not embeddable in $\BH$. A proof
can be deduced  from the well know fact that all subalgebras of
$\BH$ are Arens regular.

For $*$-algebras there are several criteria of representability in
$\BH$ and a number of sufficient conditions (see~\cite{Palmer2,
Popovych}). We will recall some of them in Section 3. In present
paper we propose a functor $\D$ from the category of associative
algebras with homomorphisms to the category of $*$-algebras with
$*$-homomorphisms. The functor $\D$ maps an associative algebra to
its $*$-double $\A * \A^*$ (see Section 2 for more details).  This
functor has been used in~\cite{BlecMod} to define maximal
$C\sp*$-algebra of an operator algebra $\A$. To be more precise, to
each completely contractive homomorphism $\theta\colon \A \to \BH$
there corresponds a $*$-homomorphism $\theta * \theta^*$ of $\A *
\A^*$ into $\BH$. Let $\J$ denote the intersection of kernels of all
$\theta * \theta^*$.  One can obtain a $C\sp*$-algebra by completing
$\A * \A^*/ \J$ with respect to the $C\sp*$-norm $\norm{x} =
\sup\norm{ \theta * \theta^*(x)}$. This $C\sp*$-algebra denoted by
$C_{max}\sp*(\A)$ is called the maximal $C\sp*$-algebra of an the
operator algebra $\A$. It should be noted that canonical map of $\A
* \A^*$ into $C_{max}\sp*(\A)$ factors trough the  canonical map
$\pi$ of $\A * \A^*$ into $\A *_{\Delta} \A^*$, the free product
amalgamated over the diagonal subalgebra $\A\cap \A^*$. Thus $\J$
contains $\ker \pi$. No conditions are known which ensure that $\J$
is trivial (however some description of $\J$ can be found
in~\cite{Duncan}). We will prove that if in the above construction
$\theta$ is not required to be completely contractive but only
completely bounded with $\norm{\theta}_{cb}\le 1+\epsilon$  then
$\J$ is always trivial (see Theorem~\ref{main}).

The main property of the functor $\D$ established in the present
paper is that an algebra $\A$ is isomorphic to a subalgebra in $\BH$
if and only if $\D(\A)$ is isomorphic to a $*$-algebra in $\BH$.  We
also give several applications of the representability result
 in the theory of operator algebras. For example
we show that each operator algebra $\A$ is completely boundedly
isomorphic to a concrete operator algebra $\B$ with the property
that each element $x$ is determined by its module $(x^*x)^{1/2}$ up
to a scalar multiple and such that $\B\cap \B^*= \Complex e$. Here
$\B^*$ is the set of adjoint operators to those in $\B$. The
subalgebra $\B\cap \B^*$ depends on the embedding of $\B$ into $\BH$
and could be characterized as a maximal $C\sp*$-subalgebra of $\B$.
If we consider all cb embeddings with cb inverses we obtain a family
of subalgebras of $\B$ of the form $\B\cap\B^*$. The above result
shows
 that the minimal subalgebra in this family is $\Complex$. In the last
section we study the opposite question, i.e. what are the maximal
 subalgebras in this family.

\section{Algebraic properties of $*$-Doubles.}

All algebras in the paper will be assumed unital and the units will
be denoted by $e$. Let $\mathcal{A}$ be an associative algebra (over
the field of complex numbers).   There exists an associative algebra
$\mathcal{A}^*$ and anti-isomorphism $\phi: \mathcal{A}\to
\mathcal{A}^*$, i.e. $\phi$ is an additive bijection and for all $a,
b\in \mathcal{A}$ and $\lambda\in \mathbb{C}$, $\phi(a b) = \phi(b)
\phi(a)$ and $\phi(\lambda a) = \overline{\lambda} a$. It is easy to
see that $\mathcal{A}^*$ is unique up to isomorphism. The existence
of $\A^*$ follows from the following construction. Consider any
co-representation $\mathcal{A} = \mathcal{F}(X)/\mathcal{I}$ where
$\mathcal{F}(X)$ is the free associative algebra with a generating
set $X$. The algebra $\mathcal{F}(X)$ is a subalgebra in the free
$*$-algebra on the generating set $X\cup X^*$. The involution of the
free  $*$-algebra maps ideal $\mathcal{I}$ to an ideal
$\mathcal{I}^*$ contained in $\mathcal{F}(X^*)$ and  the quotient
algebra $\mathcal{F}(X^*)/\mathcal{I}^*$ is clearly anti-isomorphic
to $\mathcal{A}$.

\begin{definition} The   $*$-double of $\A$, denoted by
$\D(\mathcal{A})$, is the free product $\mathcal{A}*\mathcal{A}^*$
with an involution defined on the generators by the rule $a^*=
\phi(a)$, $b^* = \phi^{-1}(b)$ where $a\in \mathcal{A}$ and $b\in
\mathcal{A}^*$ and $\phi: \mathcal{A}\to \mathcal{A}^*$ is a fixed
anti-isomorphism.
\end{definition}

It is easy to show that $*$-double of $\A$ is unique up to
$*$-isomorphism.  In sequel we will reserve  the notation  $\A*\B$
for the free product in the category of associative algebras and
will use $\A\star\B$ to denote the free product in the category of
$*$-algebras.

The $*$-double of $\mathcal{A}$ has the following universal
property. There are injective homomorphisms $i: \mathcal{A}\to
\D(\mathcal{A})$ and $j: \mathcal{A}^*\to \D(\mathcal{A})$ such that
for any $*$-algebra $\mathcal{B}$ and any homomorphism
$\pi:\mathcal{A} \to \mathcal{B}$ there exists a unique
$*$-homomorphism $\hat{\pi}: \D(\mathcal{A})\to \mathcal{B}$ such
that the following diagram is commutative:



\begin{equation}\label{diag}
\xymatrix{ \mathcal{A} \ar@{^{(}->}[r]^{i \quad } \ar[dr]_{\pi} &
\D(\mathcal{A}) \ar@{-->}[d]^{\hat{\pi}} & \mathcal{A}^*
\ar@{_{(}->}[l]_{ \quad  j} \ar[dl]^{\pi^*} \\  & \mathcal{B} }
\end{equation}

Here $\pi^*$ denotes $(\pi\circ \phi^{-1})^*$. By the universal
property of the free product every homomorphism $\rho$ between the
algebras $\A$ and $\B$ defines the homomorphism $\rho * \rho^*\colon
\A
* \A^* \to  \B * \B^*$ which coincides with $\rho$ when restricted
to the first factor and with $\rho^* = \psi\circ \rho\circ
\phi^{-1}$ when restricted to the second  one. Here $\phi\colon \A
\to \A^*$ and $\psi\colon \B\to \B^*$ are fixed anti-isomorphisms.
Below we will collect a few simple properties of $\D$.

\begin{proposition}\label{prop} Let $\A$ and $\B$ be associative algebras.
Then the following properties hold.
\begin{enumerate}
\item  If $\rho: \mathcal{A} \to \mathcal{B}$ is a homomorphism
 then $\rho * \rho^* : \D(\Amc ) \to \D(\Bmc)$ is a
$*$-homomorphism which is injective provided $\rho$ is such.
 Thus $\rho \to \rho *\rho^*$ can be considered as an  action of
the functor $\D$ on the morphisms.
\item  If $\A$ is a subalgebra in $\Cmc$ then $\D(\A)$ is a
$*$-subalgebra in $\D(\C)$.
\item   Decompose  $\A
= \Complex e \oplus \A_0$ as a direct sum of vector spaces where
 $\A_0$ is a  subspace in $\A$. Then $\A^* = \Complex\oplus \A_0^*$ and
$\D(\A)$ as a vector space is
\begin{equation} \mathbb{C}\oplus \Amc_0 \oplus \Amc_0^*\oplus
(\Amc_0\otimes \Amc_0^*) \oplus (\Amc_0^*\otimes \Amc_0) \oplus
(\Amc_0 \otimes \Amc_0^* \otimes \Amc_0)\oplus \ldots \label{Fock}
\end{equation}
\item  $\D(\A)$ is $*$-isomorphic to $\D(\A^*)$.
\item  $\D(\D(\A))$ is $*$-isomorphic to $\D(\A) \star \D(\A^*)$.
\end{enumerate}
\end{proposition}
\begin{proof}
Statements 1-4 are straightforward.  To show the last one  we will
represent $\A$ by generators and relations
$$\algebra{\set{x_\alpha}_{\alpha\in \Lambda} |
\set{p_\beta(x)=0}_{\beta\in \Delta}}.$$ Here $p(x)$ denote an
element of the free algebra $\F(\set{x_\alpha})$. For another
alphabet $y=\set{y_\alpha}$ indexed by the same set $\Lambda$ we
will denote by $p_\beta(y)$ the element of $\F(\set{y_\alpha})$
obtained by substitution $y_\alpha$ for $x_\alpha$ in $p_\beta(x)$.
For simplicity we will write $x= \set{x_\alpha}_{\alpha\in \Lambda}$
and $y=\set{y_\alpha}_{\alpha\in \Lambda}$.  In these notations
$\D(\A) = \algebra{x, x^*| p_\beta(x) = 0, p_\beta(x)^* = 0}$, and
\begin{eqnarray*}\D(\D(A))  &=& \algebra{x, y, x^*, y^*| p_\beta(x)
= 0, p_\beta(y^*)^* = 0, p_\beta(x)^*=0, p_\beta(y^*)^{**}=0}\\ &=&
\algebra{x, x^* | p_\beta(x) =0, p_\beta(x^*)=0}\star \algebra{y,
y^* | p_\beta(y^*) =0, p_\beta(y^*)^*=0}\\ &=& \D(\A) \star
\D(\A^*).
\end{eqnarray*}
\end{proof}

Fix some decomposition $\A
= \Complex e \oplus \A_0$ as in Proposition~\ref{prop}.
 The elements $c_1 c_2 \ldots
c_n$, where $c_j\in \Amc_0 \cup \Amc_0^*$, and such that $c_j\in
\Amc_0$ iff $c_{j+1}\in \Amc_0^*$ for all $j$ correspond to
elementary tensors $c_1\otimes \ldots \otimes c_n$ in decomposition
(\ref{Fock}) and will be called words.

 Fix a linear basis $\set{e_\lambda}_{\lambda \in \Lambda}$
in $\Amc_0$ (we will assume that $0$ lies in $\Lambda$ and $e_0=1$ and that
$\Lambda$ is well-ordered). The set $\Omega$ of words
\begin{equation}\label{basis} \omega= e_{i_1} e_{j_n}^* e_{i_2}
e_{j_{n-1}}^*\ldots e_{i_n} e_{j_1}^*
\end{equation}
where $i_k\not =0$ and $j_k\not= 0$ for $k\not= 1$ is a
linear basis for $\D(\Amc)$. The number $l(\omega)$ of factors not equal to $e_0$
will be called  the length of $\omega$.
 We can consider
$\Omega$ to be well-ordered by introducing homogenous lexicographic
ordering $\succeq$ on the set of multi-indices $(i_1j_n\ldots
i_nj_1)$.

Thus  arbitrary $x\in \D(\Amc)$ is a finitely supported sum $x=
\sum_{\omega\in \Omega} \alpha_\omega  \omega$ ($\alpha_\omega\in
\mathbb{C}$).
 We can define the leading monomial to be $\lambda_{\omega_0} \omega_0$ where
$\omega_0 \succeq \xi$ for all $\xi\in \Omega$ with
$\alpha_\xi\not=0$ and the leading coefficient to be
$\alpha_{\omega_0}$. Define the degree $\deg(x)$  of $x$ to be  $
l(\omega_0)$. For every $\omega\in \Omega$ we have $\omega
\omega^*\in \Omega$. This fact will be frequently used in sequel.
 Recall the following definition from~\cite{Palmer2}.
\begin{definition}
A $*$-algebra $\B$ is called ordered if for every $n$ the equation
$\sum\limits_{i=1}^n x_i x_i^* =0$ where $x_i\in \B$ have only
trivial solution $x_1= \ldots=x_n=0$.
\end{definition}


 We will prove below that all partial isometries in $\D(\A)$ are
scalar multiples of the unit. This fact will require the use of
non-commutative Gr\"obner basis theory. We refer the reader to the
appendix for necessary definitions and facts from this theory.
\begin{proposition}\label{square}
Let $\A$ be an associative algebra with the unit $e$ and  $\B=
\D(\A)$. Then the following statements hold.
\begin{enumerate}
\item  The $*$-algebra $\B$ is ordered.
\item $\B$ contains no non-scalar algebraically bounded element,
i.e. equality $\sum_{j=1}^n x_j^* x_j = e$ for some $x_j\in \B$
implies that $x_j\in \Complex e$.
\item If $x\in \A$ and $y\in \B$ are such that
$$ x^* x =y^*y$$ then $x= \lambda y$ where $\lambda\in \Complex$,
$\abs{\lambda} = 1$.
\end{enumerate}

\end{proposition}
\begin{proof}
To prove the first and the second statements take arbitrary  $x_1,
\ldots, x_n\in \B$. Let  $x_j= \sum_{k=1}^{n_j} \alpha_{j,k}
w_{j,k}$ where $w_{j,k}\in BW$ and $\alpha_{j,k}$ are complex
numbers. Denote $d= \max(\deg(x_1), \ldots,\deg(x_n))$ and
$M=\set{w_{j,k}| l(w_{j,k}) = d}$. Then for every $w\in M$ the
coefficient in front of  $w^*w$ in the decomposition of
$\sum_{j=1}^n x_j^* x_j$ into a linear combination of elements of
$BW$ is equal to $\sum_{w_{j,k}=w}\abs{\alpha_{j,k}}^2$ and hence is
not zero. In particular from this follows item 1. Moreover,
$\deg(\sum_{j=1}^n x_j^* x_j)=2d$. Since $\deg{e} = 0$ the equality
$\sum_{j=1}^n x_j^* x_j = e$ implies $d=0$. Thus $x_j\in\Complex e$.

Let us prove the last statement.
If $x=0$ then $y=0$ since  $\B$ is ordered.
 Suppose that  $x$ is not a scalar multiple of the unit $e$. We can
choose a generating set $X$ for $\A$ containing $x$ and well-ordered in such a way
 that for every $y\in X \setminus \set{x}$, $y\le x$ implies
$y=e$.   By the general theory of Gr\"obner bases there exists a
Gr\"obner basis $S\subseteq \F(X)$ for $A$. In particular, $x$ can
not be a greatest word of any $s\in S$. Then, clearly, $S\cup S^*$
is a Gr\"obner basis for $\B$. Let $BW$ denote the linear basis of
$\B$ consisting of words in $X$ (see appendix). Clearly,  the word
$x^*x$ can not appear as a term with a non-zero coefficient in
elements of $S\cup S^*$. Each of the words $x$ and $x^*$ can not be
the greatest word of any element $s\in S\cup S^*$.  Thus $x^*x\in
BW$. Decomposing $y$ by the basis we get $y= \alpha_1 w_1 + \ldots
+\alpha_n w_n$ for some $\alpha_j\in\Complex\setminus\set{0}$ and
distinct $w_j\in BW$. Hence $y^*y = \sum_{i, j}
\overline{\alpha_i}\alpha_j w_i^* w_j$. Assume that $y$ is not a
scalar multiple of $e$. Put $M=\set{j|\ l(w_j) = \deg(y)}$. For any
$i\in M$ the word $w_i^*w_i$ lies in $BW$ and has the coefficient
$\abs{\alpha_i}^2$ in the decomposition of $y^*y$ by the linear
basis $BW$. From this and from the equality $x^*x =y^*y$ follows
that $M$ is a singleton $M=\set{j}$ and $w_j = x$. In particular,
$\deg y =1$ and $y = \alpha e + \beta x$ ($\alpha, \beta
\in\Complex$). Thus we have
$$ x^*x = \abs{\alpha}^2 +\overline{\alpha}\beta x +
\overline{\beta}\alpha x^* +\abs{\beta}^2 x^*x.$$ Since $e, x, x^*,
x^*x$ lie in $BW$ we get $\alpha=0$, $\abs{\beta}=1$. Thus
$x$ and $y$ are proportional as claimed.

We are left with the case $x = \gamma e$ for some
$\gamma\in\Complex$. Let $\A^1$ denote the unitization of $\A$ which
is $\Complex\oplus\A$ as a vector space with a  multiplication given
by $(\lambda, a) (\mu,b) = (\lambda \mu, \lambda b+\mu a+ab)$. The
subalgebra $\A$ is an ideal in $\A^1$ of codimension one. Then the
fact that $x$ as an element of $\A^1$ is not a multiple of the unit
and equality $x^*x= y^*y$ is still true in $\D(\A^1)$. Hence by the
first part of the proof $x=\lambda y$ for some $y\in \Complex$.

\end{proof}

\begin{remark} A much stronger fact than statement 1
in the above theorem was established in~\cite{Popovych}. Namely, it
was proved that  $\D(\Amc)$ is $O\sp*$-representable i.e. $\D(\Amc)$
is isomorphic to a $*$-subalgebra in the $*$-algebra $L(\h)$ of all
linear operators acting on a pre-Hilbert space $\h$.
\end{remark}

\begin{corollary}
For any algebra $\A$ its $*$-double $\D(\A)$ contains no non-scalar
unitary or projection or partial isometry.
\end{corollary}
\begin{proof}
A proof can be derived from the above proposition and  the known
fact that partial isometries, projections and unitaries are
algebraically bounded elements. But we will give a short direct
proof.

Assume that $y\in \D(\A)$ is unitary then $y^*y=e^*e$. Hence by the
above proposition $y$ is scalar. If $y$ is projection then $z=2y-e$
is unitary and thus  scalar. If $y$ is partial isometry then $y^*y$
and $yy^*$ are projections which can be only 0 or e. If $y^*y=0$ or
$yy^*=0$ then $y=0$ since  $\D(\A)$ is ordered $*$-algebra.
Otherwise $y^*y=1$ and $yy^*=1$. Hence $y$ is unitary and
consequently scalar.
\end{proof}
\section{Faithful representations of $*$-algebras.}

In this section we present  conditions for a $*$-algebra to have a
faithful $*$-representation by bounded operators acting  on a
Hilbert space. A $*$-algebra with this property will be called {\it
$C\sp*$-representable}.

For a $*$-algebra $\A$ the {\it reducing ideal} denoted by
$\rad(\A)$ is
 the intersection of the kernels of all $*$-representations of
$\A$ on Hilbert spaces. There is a connection between the problem of
finding reducing ideals and finding closures of ideals in the
$C^*$-algebra of a free group. This connection was not noticed before and
  we will present it here. If a $*$-algebra $\A$ is generated by unitary elements then $\A$ is  a quotient of
the group $*$-algebra $\F_*= \Complex [\Fg(X)]$ by a $*$-ideal $\J$. Here
  $\Fg(X)$ denotes the free group with a generating set $X$. We will
identify $\A$ with $\F_*/ \J$. Let $C^*(\Fg)$ be the  group
$C\sp*$-algebra which is the universal enveloping $C\sp*$-algebra of
$\F_*$. In particular, $\F_*$ is a dense $*$-subalgebra in $C^*(\Fg)$. We will
consider $\F_*$ as a topological space with topology induced from
$C^*(\Fg)$ and  will denote by $\J^{cl}$ the closure of $\J$ in
$\F_*$.

\begin{proposition}
$\rad(\A) =  \J^{cl}/ \J$.
\end{proposition}
\begin{proof}
Let $q\colon \F_* \to \F_*/ \J$ be the canonical epimorphism. Let
$\pi\colon \F_*/ \J \to C^*(\F_*/ \J)$  denote  the canonical
$*$-homomorphism into enveloping $C\sp*$-algebra $C^*(\F_*/ \J)$. We
 will consider $\pi$ as a $*$-representation $\pi\colon \F_*/ \J
\to \BH$ on some Hilberst space $\h$. From the universal property of
the enveloping $C^*$-algebra follows that $\rad(\A) = \ker\pi$.
Clearly, $\pi\circ q$ is a $*$-representation of $\F_*$. Thus by the
universal property of the enveloping $C\sp*$-algebra $\pi\circ q$
can be extended to the $*$-representation of $C^*(\F_*)$. Hence
there is an extension of $q$ to a surjective $*$-homomorphism
$\hat{q}\colon C^*(\F_*) \to C^*(\F_*/ \J)$. Thus $\J\subseteq \ker
\hat{q}$. Since $\ker \hat{q}$ is closed the closure
$\overline{\J}$ of $\J$ in $C^*(\F_*)$ is also contained in $\ker
\hat{q}$.

To show the converse inclusion note that the quotient $C^*(\F_*)/
\overline{\J}$ is a $C\sp*$-algebra. It can be regarded as a
$C\sp*$-subalgebra in $\BOP(\widetilde{\h})$ for a Hilbert space
$\widetilde{\h}$. Hence the quotient map $C^*(\F_*)\to C^*(\F_*)/
\overline{\J}$ can be viewed as a $*$-representation $\tau\colon
C^*(\F_*)\to \BOP(\widetilde{\h})$. Restriction of $\tau$ to $\F_*$
 maps  $\J$ to zero and thus can be regarded as a $*$-representation
$\widetilde{\tau}$ of the quotient $\F_*/ \J$. We will denote by the
same symbol its unique extension to the $C^*$-algebra  $C^*(\F_*/
\J)$. A moment reflection reveals that the following diagram is
commutative
\begin{equation*}\label{diag}
\xymatrix{ C^*(\F_*) \ar[r]^{\hat{q} } \ar[dr]_{\tau} & C^*(\F_*/
\J) \ar[d]^{\widetilde{\tau}} &  \\  & \BOP(\widetilde{\h}) }
\end{equation*}
Thus $\ker \hat{q} \subseteq \ker \tau = \overline{\J}$ which gives
 $\rad (\F_*/ \J) = (\overline{\J}\cap \F_*)/\J$.
\end{proof}
\begin{remark}The similar description can be obtained for any algebraically
bounded $*$-algebra (for the definition see for
example~\cite{Popovych}) with a $C\sp*$-algebra generated by a free
family of contractions in place of $C^*(\F_*)$.
\end{remark}

Let us recall a criterium for a $*$-algebra to have a trivial
reducing ideal in terms of cones of self-adjoint elements.

Firstly we give necessary definitions and fix notations. Let
$\mathcal{A}_{sa}$ denote the set of self-adjoint elements in
$\mathcal{A}$. A subset $C\subset \mathcal{A}_{sa}$ containing the
unit $e$ of $\mathcal{A}$ is an \textit{algebraically admissible} cone
(see \cite{powers}) provided that
\begin{enumerate}[(i)]
\item   $C$ is a cone  in $\mathcal{A}_{sa}$, i.e. $\lambda
x+\beta y\in C$ for all $x$, $y\in C$ and $\lambda\geq 0$,
$\beta\geq 0$, $\lambda,\beta\in\mathbb{R}$; \item $C \cap (-C)
=\{0\}$; \item $x C x^* \subseteq C$ for every $x\in \mathcal{A}$.
\end{enumerate}

We call $e\in \mathcal{A}_{sa}$ an \textit{order unit} if for every
$x\in \mathcal{A}_{sa}$ there exists $r>0$ such that $re+x\in C$. An
order unit $e$ is \textit{Archimedean} if $re+x\in C$ for all $r>0$
implies that $x\in C$. The following theorem was proved
in~\cite{JushPop}.
\begin{theorem}\label{onecone}
Let $\mathcal{A}$ be a  $*$-algebra with the unit $e$ and
$C\subseteq \mathcal{A}_{sa}$ be a cone containing $e$. If $x C x^*
\subseteq C$ for every $x\in \mathcal{A}$ and $e$ is an Archimedean
order unit then there is a  unital $*$-representation $\pi:
\mathcal{A} \to \BH$ such that $\pi(C) = \pi(\mathcal{A}_{sa})\cap
\BH^+$ where $\BH^+$ is the set of positive operators. Moreover
\begin{enumerate} \item\label{onecone1} $\norm{\pi(x)} = \inf \{r>0 : r^2 \pm x^*x\in C\}$.
\item\label{onecone2} $\ker \pi =\{ x : x^*x \in C \cap (-C)\}$.
\item\label{onecone3} If $C \cap (-C) = \{0\}$ then $\ker \pi =\{
0\}$, $\norm{\pi(a)} = \inf \{r>0 : r \pm a\in C\}$ for all
$a=a^*\in \mathcal{A}$ and $\pi(C) = \pi(\mathcal{A})\cap \BH^+$.
\end{enumerate}
\end{theorem}
In particular a unital $*$-algebra $\mathcal{A}$ which possess  an
algebraically admissible cone such that the algebra unit is an
Archimedean order unit has a faithful $*$-representation by bounded
operators on a Hilbert space.

In the next section we will use in essential way the free product
construction.  Let $\mathcal{A}_1$ and $\mathcal{A}_2$ be two unital
$*$-algebras and $\phi_1$, $\phi_2$ be linear unital functionals on
$\mathcal{A}_1$  and  $\mathcal{A}_2$ respectively. Let
$\overset{\circ}{\mathcal{A}_j} = \ker
 \phi_j$ ($j=1,2$). The algebraic free product $\mathcal{A}_1 \star
 \mathcal{A}_2$ as a linear space is a quotient of the space
 $$ \mathbb{C} \oplus {\mathcal{A}_1} \oplus{\mathcal{A}_2}
 \oplus ({\mathcal{A}_1} \otimes {\mathcal{A}_2})
 \oplus ({\mathcal{A}_2}\otimes {\mathcal{A}_1})
 \oplus ({\mathcal{A}_1}\otimes{\mathcal{A}_2}
 \otimes {\mathcal{A}_1})\oplus\ldots $$
 by a subspace in order to identify the units in $\mathbb{C}$,
 $\mathcal{A}_1$ and $\mathcal{A}_2$. As a vector spaces
 $\mathcal{A}_j = \mathbb{C}\oplus \overset{\circ}{\mathcal{A}_j}$
 ($j=1,2$) and thus as vector space
$$\mathcal{A}_1 \star
 \mathcal{A}_2 =  \mathbb{C} \oplus \overset{\circ}{\mathcal{A}_1} \oplus \overset{\circ}{\mathcal{A}_2}
 \oplus (\overset{\circ}{\mathcal{A}_1} \otimes \overset{\circ}{\mathcal{A}_2})
 \oplus (\overset{\circ}{\mathcal{A}_2}\otimes \overset{\circ}{\mathcal{A}_1})
 \oplus (\overset{\circ}{\mathcal{A}_1}\otimes\overset{\circ}{\mathcal{A}_2}
 \otimes \overset{\circ}{\mathcal{A}_1})\oplus\ldots $$
 The projection onto $\mathbb{C}$ associated with the above direct
 sum is a free product $\phi_1\star \phi_2$.
It is known that the free product of two faithful states is a
faithful one (see \cite{Avitzour}) which implies that the free
product of two $C^*$-algebras is $C^*$-representable. Another proof
can be deduced from Theorem~\ref{onecone} by defining an
algebraically admissible cone as $C(\mathcal{A}_1,\mathcal{A}_2)=
\{a\in (\mathcal{A}_1 \star
 \mathcal{A}_2)^+ | \text{ for all } \phi_1\in S(\A_1),
 \phi_2\in S(\mathcal{A}_2)$,  and every $x\in \mathcal{A}_1 \star
 \mathcal{A}_2,  (\phi_1\star \phi_2)(x a x^*)\ge 0\}.$ Where
 $\mathcal{A}_1$ and $\mathcal{A}_2$ are $C^*$-algebras,
 $S(\A_j)$ is the set of states on $\A_j$ and
 $(\mathcal{A}_1 \star
 \mathcal{A}_2)^+$ is the set of finite sums of elements of the form $x^*x$ with
$x\in \mathcal{A}_1 \star
 \mathcal{A}_2$.

\section{Faithful Representations of $*$-Doubles.}\label{faithdb}

In this section we will prove that each operator algebra is
completely boundedly isomorphic to an operator algebra which
generates its $*$-double. In particular, this enables one to recast
existing criteria of $C\sp*$-representability of $*$-algebras (for
example those presented in the previous section) as  criteria of
representability of associative algebras in $\BH$.

We are going to prove that the $*$-double of every operator algebra
is representable in $\BH$ but first we need the following.

\begin{lemma}\label{alg}
Let $\A$ be a $*$-algebra. Then  $\D(\A)$ is $*$-isomorphic to a
$*$-subalgebra in the free product $\A\star \D(\Complex
[\mathbb{Z}_2])\star \D(\Complex [\mathbb{Z}_2])$
\end{lemma}
\begin{proof}
  Fix some presentation by
generators and relations $$\mathcal{A} = \mathbb{C}\langle x, x^* |
p_\beta (x, x^*) = 0, \beta\in \Delta \rangle.$$ Here $x
=\set{x_\alpha}_{\alpha\in \Lambda}$, $x^* =
\set{x^*_\alpha}_{\alpha\in \Lambda}$.

Then, clearly,  $\mathcal{A}^* =   \mathbb{C}\langle y, y^*| p_\beta
(y, y^*) = 0, \beta\in \Delta \rangle$ ($\simeq \A$) with
anti-isomorphism between $\mathcal{A}$ and $\mathcal{A}^*$ given by
$x_\alpha \mapsto y_\alpha^*$. Hence
$$\D(\mathcal{A}) = \mathbb{C}\langle x, x^*,
y, y^* | p_\beta (x, x^*) = 0, p_\beta (y, y^*) = 0, \beta\in \Delta
\rangle. $$ The involution on $\D(\mathcal{A})$ is given by
$x_\alpha^\sharp = y_\alpha^*$, $(x^*_\alpha)^\sharp = y_\alpha$.

Let $\mathcal{G}= \Complex [\mathbb{Z}]$ be the group algebra of the
group $\mathbb{Z}$ and $g$ its invertible generator.
 Let us define a
map $\gamma$ on the set of generators of $\D(\mathcal{A})$ with
values in the free product of $*$-algebras $\mathcal{A} \star
\D(\mathcal{G})$  by the rule
$$\gamma(x_\alpha) = g^{-1} x_\alpha g, \gamma(x_\alpha^*) = g^{-1}
x_\alpha^* g, \gamma(y_\alpha^*) = g^* x_\alpha^* g^{*-1},
\gamma(y_\alpha) = g^* x_\alpha g^{*-1}.$$

Then \begin{eqnarray*} p_\beta(\gamma(x), \gamma(x^*)) &=& g^{-1}
p_\beta (x,x^*) g,\\  p_\beta(\gamma(y), \gamma(y^*)) &=& g^*
p_\beta (x,x^*)g^{*-1}. \end{eqnarray*} Thus $\gamma$ can be
extended to a $*$-homomorphism between $\D(\mathcal{A})$
 and $\mathcal{B} = \mathcal{A} \star \D(\mathcal{G})$. Clearly, the image of
 $\gamma$ is a $*$-subalgebra  of $\mathcal{B}$ generated by $g^{-1} a
 g$ and $g^* a g^{*-1}$ where $a\in \mathcal{A}$. Denote this image by
 $\mathcal{D}$. We will  show   that $\gamma$ is injective.

\ Let $S\subset \F(X)$ (where $X= \set{x_\alpha}_\alpha \cup
\set{x^*_\alpha}_\alpha$) be a Gr\"obner basis for
 $\mathcal{A}$ and $W_{\mathcal{A}}$ be the set $BW(S)$ of basis words
 corresponding to $S$ (see appendix). Then the set of basis words $W_{\mathcal{B}}$ corresponding to
 $\mathcal{B}$ will be exactly the set of words in generators
 $x_\alpha$, $x_\alpha^*$, $g$, $g^{-1}$, $g^*$ and $g^{*-1}$
 which  contain no words $\hat{s}$ ($s\in S$) and no words $g g^{-1}$, $g^{-1} g$, $g^* g^{*-1}$,
 $g^{*-1} g^*$ as sub-words. Hence the word
  \begin{gather}
g^{-1} u_1 g g^* w_1 g^{*-1} \ldots g^{-1} u_n g g^* w_n g^{*-1}
\tag{*}
 \end{gather}
 is in $W_{\mathcal{B}}$ whenever $u_j$, $w_j\in W_{\mathcal{A}}$ ($1\le j\le
 n$), $u_k\not= e$ for $k\ge 2$ and $w_s \not= e$ for $s\le n-1$.

 Using the relations $p_\beta(g^{-1}x g, g^{-1}x^* g) =0$ and $p_\beta(g^{*} x g^{*-1}, g^{*} x^* g^{*-1})
 =0$  each polynomial in generators $g^{-1}x g$, $g^{-1}x^* g$,
 $g^{*} x g^{*-1}$, $g^{*} x^* g^{*-1}$ can be reduced to a linear
 combination of the words of the form $(*)$ which is a canonical
 form in $\mathcal{B}$. Indeed, to see this it is enough
to consider words in generators $g^{-1}x g$, $g^{-1}x^* g$,
 $g^{*} x g^{*-1}$, $g^{*} x^* g^{*-1}$ and in this case the claim is obvious.  Hence
 $\mathcal{D} \simeq \mathbb{C} \langle \set{a_\alpha, b_\alpha, a_\alpha^*, b_\alpha^*}_\alpha | p_\beta(a, b) =0,
 p_\beta(b^*, a^*) =0, \beta\in \Delta \rangle$ via isomorphism $a_\alpha \mapsto g^{-1}x_\alpha g$,
 $b_\alpha \mapsto g^{-1}x_\alpha^* g$, $a_\alpha^* \mapsto g^{*} x_\alpha^* g^{*-1}$,
 $b_\alpha^* \mapsto g^{*} x_\alpha g^{*-1}$.  This proves that map $\phi$ given on generators as
\begin{eqnarray*}\phi(g^{-1}x_\alpha g) &=&
 x_\alpha,\
\phi(g^{-1}x_\alpha^* g)=
 x_\alpha^*,\\  \phi(g^* x_\alpha g^{*-1})& = &
 y_\alpha,\  \phi(g^* x_\alpha^* g^{*-1})=
 y_\alpha^*
\end{eqnarray*} can be extended to a homomorphism between $\mathcal{D}$
 and $\D(\mathcal{A})$. Clearly, $\phi$ is inverse to $\gamma$.
 Thus $\D(\mathcal{A})$ can be regarded as a $*$-subalgebra in the
 free product of $*$-algebras $\mathcal{A} \star \D(\mathcal{G})$.

The  $*$-algebra $ \D(\mathcal{G})$ is isomorphic to a
$*$-subalgebra generated by $s_1 s_2$ in the free product
 $\D(\mathbb{C}\langle s_1| s_1^2 =1 \rangle) \star \D(\mathbb{C}\langle s_1| s_1^2 =1
 \rangle)$. Clearly,
$\mathbb{C}\langle s_1| s_1^2 =1 \rangle$ is isomorphic to
$\Complex[\mathbb{Z}_2]$ and thus $\mathcal{A} \star
\D(\mathcal{G})$ is $*$-isomorphic to a $*$-subalgebra in $\A\star
\Complex[\mathbb{Z}_2] \star \Complex[\mathbb{Z}_2]$.
\end{proof}

The main result of this section is the following.
\begin{theorem}\label{main}
Let $\A$ be an associative algebra. Then there is a pre-Hilbert
space $\h$ and an injective $*$-homomorphism $\phi$ from $\D(\A)$
into the $*$-algebra $\LH$ of the linear operators on $\h$.

If $\A$ is an operator algebra then $\h$ can be chosen to be a
Hilbert space. More precisely,  there exists an operator algebra
$\A_1\subseteq \BH$,  a completely isometric isomorphism $\phi\colon
\A\to \A_1$ and an invertible $S\in\BH$ such that $\D(\A)$ is
$*$-isomorphic to the $*$-subalgebra generated by $S^{-1} \A_1 S$.
Moreover, for every $\eps>0$ operator $S$ can be chosen s.t.
$\norm{S^{-1}} \norm{S} = 1+\eps$.
\end{theorem}
\begin{proof}
The first statement is equivalent to $\D(\A)$ being
$O\sp*$-representable which was proved in~\cite{Popovych}.

If $\A$ is an operator algebra then $\A$ is completely isometrically
isomorphic to a subalgebra in a $C\sp*$-algebra $\B$. By
Lemma~\ref{alg}, $*$-algebra $\D(\A)$ is $*$-isomorphic to
$*$-subalgebra in $\B\star \D(\Complex [\mathbb{Z}_2]) \star
\D(\Complex [\mathbb{Z}_2])$. The $*$-algebra $\D(\Complex
[\mathbb{Z}_2])\simeq \Complex \seq{s,s^*| s^2=s^{*2} =e}$ is
$C\sp*$-representable. Indeed, the linear basis of this algebra
consists of the words $e, s, s^*, (ss^*)^k, (s^*s)^k$ where $k\ge
1$. Decomposing arbitrary element by the above basis it is easy to
check that the direct sum of two dimensional representations of the
form
\[s\to
\left(                                                 \begin{array}{cc}
                                                   1 & \lambda \\
                                                   0 & -1 \\
                                                 \end{array}
                                               \right),
\] where $\abs{\lambda}\le \eps$ (for any fixed $\eps>0$) is a faithful $*$-representation.
Thus $\D(\Complex [\mathbb{Z}_2])$ is $C\sp*$-representable. Hence
the free product $\B\star \D(\Complex [\mathbb{Z}_2]) \star
\D(\Complex [\mathbb{Z}_2])$ is also $C\sp*$-representable.

Let  $\pi\colon \B\star \D(\Complex [\mathbb{Z}_2]) \star
\D(\Complex [\mathbb{Z}_2]) \to \BH$ be a faithful
$*$-representation. The algebra  $\A_1 = \pi(\A)$ is completely
isometrically isomorphic to $\A$. Note that under $*$-embedding of
$\D(\A)$ into $\B\star \D(\Complex [\mathbb{Z}_2]) \star \D(\Complex
[\mathbb{Z}_2])$ described in Lemma~\ref{alg} the algebra $\A$ is
mapped onto $(s_1s_2)^{-1} \A (s_1s_2)$. Setting  $S =\pi(s_1s_2)$
we have that the operator algebra $ S^{-1} \A_1 S $ generates
$*$-subalgebra $*$-isomorphic to $\D(\A)$.
\end{proof}

\begin{corollary}\label{link}
An associative algebra $\A$ is isomorphic to a subalgebra in a
$C\sp*$-algebra  if and only if $\D(\A)$ is isomorphic to a
$*$-subalgebra in a $C\sp*$-algebra.
\end{corollary}
\begin{corollary}
Every unital operator algebra $\A$ is cb isomorphic to an operator
algebra $\B\subset \BH$ such that \begin{enumerate} \item  $\B\cap
\B^* =\Complex e$ (here $\B^*$ is the set of adjoint operators to elements of $\B$).
\item $*$-Algebra generated by $\B$ contains no non-scalar partial isometries.
 \item  If $\abs{b_1} =\abs{b_2}$ for
some $b_1$, $b_2 \in \B\setminus \set{0}$ then  $b_1 = \lambda
b_2$ for some $\lambda \in \Complex$. Here $\abs{b} = \sqrt{b^*b}$.
\end{enumerate}
\end{corollary}
\begin{proof}
Let $\B$ be as in Theorem~\ref{main}. Then, clearly, $\B\cap\B^* =
\Complex e$. The second statement follows directly from
Proposition~\ref{square}.
\end{proof}
\begin{remark}
In Theorem~\ref{main} we have shown  that $\D(\A)$ can be considered
as a $*$-subalgebra in $\A \star \D(\Complex [\mathbb{Z}])$. The
problem of extending $*$-homomorphisms into $\BH$  defined on this subalgebra to a $*$-homomorphism of $\A \star \D(\Complex
[\mathbb{Z}])$ is equivalent to the Kadison's Similarity problem.

In 1955 R. Kadison raised the following problem. Is any bounded
homomorphism $\pi$ of a $C^*$-algebra $\mathcal{A}$ into $\BH$
similar to a $*$-representation? The similarity above means that
there exists an invertible operator $S\in \BH$ such that $x\to
S^{-1}\pi(x) S$ is a $*$-representation of $\mathcal{A}$.

 The affirmative answer to the above question is equivalent to the statement that for every $C\sp*$-algebra $\A$ and every $*$-representation $\pi$ of
$\D(\A)$ such that $\pi$ is bounded when restricted to $\A$, $\pi$
has an extension to $*$-representation of $\A\star \D(\Complex
[\mathbb{Z}])$.

Let us note that Kadison's Similarity problem has affirmative answer
for nuclear algebras $\A$ or in the case when $\pi|_\A$
has a cyclic vector (see~\cite{Bunce,Haagerup}). Further results and reformulations of
Kadison's similarity problem can be found in~\cite{Pisier, Hadwin, Jush}.
\end{remark}

One of the main approaches to study an operator algebra $\A$
(initiated by W. Arveson~\cite{Arveson}) prescribes to study those
properties of the $C\sp*$-algebra $C^*(\A)$ generated by the image
of $\A$ under completely isometric embedding $\A\hookrightarrow\BH$
which are invariant with respect to such  embeddings.

However, the $C\sp*$-algebras $C\sp*(\A)$ themselves can be quite
different depending on the embeddings the operator system $\A+ \A^*$
and, in particular, $C\sp*$-algebra $\A\cap \A^*$ are invariants of
completely isometric embeddings. The situation becomes different if
we consider completely bounded embeddings. Theorem~\ref{main} shows
that for every operator algebra $\A$ there exists a completely
bounded embedding (with cb inverse) such that the image $\A_1$ has
trivial intersection with $\A_1^*$, i.e. $\A_1\cap\A_1^* = \Complex
e$.
 We will use this observation to show that in the following theorem one
can not replace completely contractive homomorphism  with completely
isometric one in case of embeddings.

The following result can be found in~\cite{Paulsen}.
\begin{thmp}

Let $\B$ be an operator algebra. Assume that a unital homomorphism
$\rho: \B \to \BH$ is completely bounded. Then there exists
invertible $S\in \BH$ s.t. $S^{-1}\rho(\cdot) S$ is completely
contractive.
\end{thmp}

If in the above theorem $\B$ is a $C^*$-algebra and $\rho$ is
injective then $S^{-1}\rho(\cdot) S$ is necessarily completely
isometric. The following example shows that this is no longer true
for general operator algebras.

\begin{example}
Let $\A$ be the algebra of lower triangular matrices $T_n(\Complex)$
($n\ge 2$). By Theorem~\ref{main} there exists an operator algebra
$\B\in\BH$  s.t. $\B\cap \B^*= \Complex e$ and a completely bounded
unital isomorphism $\phi\colon \B\to\A$ with cb inverse. We want to
show that there is no invertible $S\in M_n(\Complex)$ s.t.
$S^{-1}\phi(\cdot) S$ is completely isometric. If such $S$ exists
then $$(S^{-1}\A S)\cap (S^{-1}\A S)^* = \Complex e.$$

Take a matrix $Z\in (S^{-1}\A S)\cap (S^{-1}\A S)^*$. Then there are
lower triangular matrices $A$, $B$ s.t. $Z = S^{-1} A S = S^* B^*
S^{*-1}$. Note that $Z$ is a scalar multiple of the identity matrix
iff $A=B^*$ and $A$ is a scalar multiple of the identity matrix.
With $C=S^*S$ we have \begin{equation}\label{eqv}
 CB^* C^{-1} = A.
\end{equation}
Consider Cholesky decomposition of the positive matrix $C$,  $C=
LL^*$ where $L$ is a lower triangular matrix with positive elements
on the diagonal. Equation~(\ref{eqv}) is equivalent to $$ L^* B^*
L^{*-1} = L^{-1}A L.$$
 Take arbitrary diagonal matrix $T =\diag(d_1, \ldots, d_n)$ with at
least two distinct entries. Then the lower triangular matrices $A=
LTL^{-1}$ and $B=L T^*L^{-1}$ are not scalar multiples of the
identity matrix and they satisfy~(\ref{eqv}). Thus $(S^{-1}\A S)\cap
(S^{-1}\A S)^* \not=\Complex e$ for any $S$.
\end{example}
 In the above example we can consider bounded lower triangular
operators on infinite dimensional Hilbert space instead of
$T_n(\Complex)$.

Another consequence of Theorem~\ref{main} is that the $*$-doubles of
$C\sp*$-algebras possesses two natural structures of
$C\sp*$-algebras.
 Fix an anti-homomorphism $\phi\colon \A\to \A^*$.
If $\A$ is a $*$-algebra then $\A * \A^*$ possesses two involutions.
The first one is given by $a^\sharp =\phi(a)$, $b^\sharp =
\phi^{-1}(b)$ for $a\in \A$ and $b\in \A^*$ and the second one by
$a^\star = a^*$, $b^\star = \phi(\phi^{-1}(b)^*)$. With the first
involution $\A * \A^*$ is the $*$-double $\D(\A)$ whereas with the
second one we get the free product of $*$-algebras  $\A \star \A$.
Let $\psi\colon\A
* \A^* \to \A
* \A^*$ be an isomorphism such that $\psi(a) = \phi(a^*)$, $\psi(b)
= \phi^{-1}(b)^*$. By direct verification we get the following
proposition.
\begin{proposition}
The following properties hold.
\begin{enumerate}
\item For all $x\in \A * \A^*$, $ \psi(x^\sharp) = x^\star$, $\psi(x^\star) =
x^\sharp$.
\item $\psi \circ \psi = id$.
\item $\psi$ is $*$-automorphism of $\A * \A^*$ with respect to each of the involutions $\sharp$ and $\star$.
\end{enumerate}
\end{proposition}

\begin{proposition}
Let $\A$ be a $C\sp*$-algebra and  $\B= \A * \A^*$. Then there is a
pre-$C\sp*$-norm $\norm{\cdot}_1$ on $(\B, \star)$ and
pre-$C\sp*$-norm $\norm{\cdot}_2$ on $(\B, \sharp)$ such that for
all $x\in\B$, $\norm{x^\star}_1 = \norm{x^\sharp}_1 =\norm{x}_1$ and
$\norm{x^\star}_2 = \norm{x^\sharp}_2 =\norm{x}_2$.
\end{proposition}
\begin{proof}
Clearly $(\A * \A^*, \star)$ is $*$-isomorphic to the free product
$\A \star \A$ of $*$-algebras. The $*$-algebra $\A \star \A$ is
$*$-isomorphic to a $*$-subalgebra in some $\BH$ as  noted at the
end of section \ref{faithdb}. Moreover, since any $C\sp*$-algebra is
generated by the unitary elements the same holds for $\A\star\A$ and
hence there exists the universal enveloping $C\sp*$-algebra
$C\sp*(\B)$ of $(\B, \star)$. The canonical homomorphism from $(\B,
\star)$ to $C\sp*(\B)$ is injective and thus induces a
pre-$C\sp*$-norm $\norm{\cdot}_1$ on $(\B, \star)$. Since $\psi$ is
a $*$-automorphism of $(\B, \star)$ it extends to a $*$-automorphism
of $C\sp*(\B)$ and hence is isometric. From this follows that
$\norm{x^\sharp}_1= \norm{\psi(x^\star)}_1 =\norm{x}_1$ for all
$x\in \B$.

By Theorem~\ref{main} there is a faithful $*$-representation $\pi:
(\B,\sharp)\to \BH$. Clearly, $\pi\circ\psi$ is also
$*$-representation. Hence $\norm{b}_2 = \max(\norm{\pi(b)},
\norm{\pi(\psi(b))})$ is pre-$C\sp*$-norm on $(\B,\sharp)$.
Moreover, \begin{eqnarray*} \norm{b^\star}_2 &=& \norm{\psi(b^*)}_2
 =\max(\norm{\pi(\psi(b^*))}, \norm{\pi(\psi\circ\psi(b^*))})
\\ &=&\norm{b^*}_2
\end{eqnarray*} since $\psi\circ\psi = id$.
\end{proof}

\section{The maximal $C^*$-subalgebras of operator algebras.}
Let $\A$ be an operator algebra. If $\phi\colon\A\to \B$  is an
injective homomorphism onto the operator algebra $\B\subseteq \BH$
such that $\norm{\phi}_{cb}<\infty$ and
$\norm{\phi^{-1}}_{cb}<\infty$ then $\A$ and $\B$ will be called cb
isomorphic. Subalgebra $\B\cap \B^*$ is a $C^*$-algebra and we
denote by $\A_{\phi}$ its pre-image $\phi^{-1}(\B\cap \B^*)$. Denote
by $\Phi(\A)$ the set of all such $\phi$ and by $\Phi(\A)_r$ its
subset of those $\phi$ such that $\norm{\phi}_{cb}\le r$ and $
\norm{\phi^{-1}}_{cb}\le r$.

We are interested in a description of all possible subalgebras of
$\A$ of the form $\A_\phi$. The following proposition shows that a
description of the maximal subalgebras of this form should be
important.

\begin{proposition}
If $\A\subseteq \BH$ is an operator algebra then for every von
Neumann algebra $\mathcal{W}$ of $\BH$  and every $r>1$ there exists
$\phi\in \Phi_r(\A)$ such that $$\A_\phi = \A\cap \A^* \cap
\mathcal{W}.$$
\end{proposition}
\begin{proof}
Fix $r>1$. Consider a collection $\set{S_\alpha}\subseteq \BH$ such
that $\norm{S_\alpha} \norm{S_\alpha^{-1}} \le r$ for every $\alpha$
and put $\mathcal{K} = \h\oplus \oplus_{\alpha} \h_\alpha$ where $
\h_\alpha$ is a copy of $\h$. Define $$ \phi(x) = x\oplus
\oplus_{\alpha}S_\alpha^{-1} x S_\alpha.$$ For $x\in \A$ the
operator $\phi(x)$ is self-adjoint iff $x^*=x$ and $C_\alpha x = x
C_\alpha$ for all $\alpha$ and $C_\alpha = S_\alpha S_\alpha^*$.
Since $\phi(\A)\cap \phi(\A)^*$ is a linear span of its self-adjoint
elements we get $\A_\phi = \A\cap \A^* \cap \set{C_\alpha}'$ where
$\set{C_\alpha}'$ denotes the commutant of $\set{C_\alpha}$.

For any von Neumann algebra $\W$ its commutant $\W'$ is a von
Neumann algebra and thus it is generated by its self-adjoint
elements $C$ such that $\frac{1}{r}I \le C\le r^2 I$. Hence $\W' =
\set{C_\alpha}_\alpha$ for some family such that $C_\alpha =
S_\alpha S_\alpha^{*}$ and $\norm{S_\alpha} \norm{S_\alpha^{-1}} \le
r$. By von Neumann bicommutant theorem $\set{C_\alpha}'_\alpha =\W$.
Hence $\A_\phi = \A\cap \A^* \cap \W$.
\end{proof}

A maximal subalgebra $A_\psi$ in the family $\set{\A_\phi| \phi\in
\Phi(\A)}$ will be called $\Phi$-maximal and in the family
$\set{\A_\phi| \phi\in \Phi_r(\A)}$ will be called $\Phi_r$-maximal.
The homomorphism  $\psi$ as above will be called $\Phi$-maximal or
respectively $\Phi_r$-maximal. The existence of $\Phi$- or
$\Phi_r$-maximal subalgebras is not obvious. If $\A$ is finite
dimensional then clearly $\Phi$-maximal and $\Phi_r$-maximal
algebras exists. In infinite dimensional case we have the following:
\begin{theorem}\label{ultra}
Let $\A$  be a separable operator algebra,  $r>0$  and $\psi\in
\Phi_r(\A)$ then there exists $\phi\in \Phi_r(\A)$ such that
$\A_\psi \subseteq \A_\phi$ and $\A_\phi$ is $\Phi_r$-maximal.
\end{theorem}
\begin{proof}
By Zorn's lemma and because $\A$ is separable it is suffices to show
that for any increasing sequence $\A_{\phi_n}\subseteq
\A_{\phi_{n+1}}$, where $\phi_k \in \Phi_r(\A)$ there exists
$\phi\in  \Phi_r(\A)$ such that $\A_{\phi_n}\subseteq \A_{\phi}$ for
all $n\ge 1$.

Let $\B_n = \phi_n(\A)$ and $\C_n = C^*(\B_n)$. Let $\F$ be any free
ultrafilter on $\mathbb{N}$. Consider ultraproduct $\mathcal{D} =
\prod_{\F} \C_n$ which is a quotient of $l^\infty$-direct product
$\prod_{i=1}^\infty \C_n$ by the ideal $\J = \set{(x_n)_{n=1}^\infty
| \lim \limits_{\F} \norm{x_n} =0 }$. It is known that ultraproducts of
 $C\sp*$-algebras are again $C\sp*$-algebras. Consider the homomorphism
$\psi\colon \A \to \mathcal{D}$ given by $\phi(a) =
(\phi_n(a))_{n=1}^\infty$. Since for every $x= (x_n)_{n=1}^\infty\in
\mathcal{D}$, $\norm{x} = \lim_{\F} \norm{x_n}$ passing to the limit
in the inequalities $\norm{\phi_n^{(m)}([a_{ij}]_{ij})} \le r \norm{
[a_{ij}]_{ij} }$ for $[a_{ij}]_{ij} \in M_m(\A)$ and
$\norm{(\phi_n^{-1})^{(m)}([b_{ij}]_{ij})} \le r \norm{
[b_{ij}]_{ij} }$ for $[b_{ij}]_{ij} \in M_n(\B)$ we get
$\phi^{(m)}([a_{ij}]_{ij})\le \norm{[a_{ij}]_{ij}}$ and
$\norm{(\phi^{-1})^{(m)}([\phi(a_{ij})]_{ij})} \le r \norm{
[\phi(a_{ij})]_{ij}}$. Thus $\phi\in \Phi_r(\A)$.

We will show that $\A_{\phi_n} \subseteq \A_\phi$ for all $n\ge 1$.
Since $\A_{\phi_n}$ is a linear span of those $a\in \A$ such that
$\phi_n(a)^* = \phi_n(a)$ it is suffices to show for such $a\in \A$
that $\phi(a)$ is self-adjoint. Recall that an element $x$ in a
$C\sp*$-algebra is self-adjoint if and only if for every
$\alpha\in\mathbb{R}$, $\norm{exp(i \alpha x)} = 1$ (see~\cite[Prop.
44.1]{DoranBelfi}). Since $\phi_n(a)^* =  \phi_n(a)$ and
$\A_{\phi_k}$ is an increasing family  we have $\norm{exp(i \alpha
\phi_m(a))} = 1$ for all $m\ge n$. Thus
$$\norm{exp(i \alpha \phi(a))} =\lim_{\F}\norm{exp(i \alpha
\phi_m(a))} = 1.  $$ Hence $\phi(a)^* = \phi(a)$ and $\A_{\phi_n}
\subseteq \A_\phi$.
\end{proof}

If $\A$ is an operator algebra in $\BH$ and $\phi:\A\to \B$ lies in
$\Phi(\A)$ then $\phi^{-1}: \B\cap \B^*\to \A_\phi$ is a cb
homomorphism of $C\sp*$-algebra $ \B\cap \B^*$. By the famous
Haagerup theorem (see~\cite{Haagerup}) there is an invertible $S\in
\BH$ such that $S^{-1}\phi(\cdot)S$ is a $*$-representation and
hence $S^{-1}\A_\phi S$ is a $C\sp*$-subalgebra in $\BH$. Thus for
every $\phi\in \Phi(\A)$ there exists an invertible $S\in \BH$ such
that $\A_\phi\subseteq \A_{Ad S}$. In particular every
$\Phi$-maximal  or $\Phi_r$-maximal subalgebra  $\A_\phi$ is of the
form $\A_{Ad S}$. However the subalgebra $\Complex e$ in the
following example is of the form $\A_\phi$ with $\phi\in \Phi(\A)$
but not of the form $\A_{Ad S}$.

\begin{example}\label{examp}
Let $\A= \set{\left(
                \begin{array}{cc}
                  \alpha & \beta \\
                  0 & \gamma \\
                \end{array}
              \right)| \alpha, \beta, \gamma \in \Complex
}$ be an operator algebra in $M_2(\Complex)$. For every $c\in
\Complex\setminus \set{0}$ consider the subalgebra $$V_c=
\set{\left(
                                                            \begin{array}{cc}
                                                              y & x \\
                                                              0 & y+c x \\
                                                            \end{array}
                                                          \right)
| x, y\in \Complex}.$$ Let $V_0$ denote the algebra of diagonal
matrices in $M_2$.  The algebras $V_c$ are clearly commutative.
Since $\A$ is not semisimple $\A$ is not of the form $\A_\phi$.
Hence any two-dimensional subalgebra of the form $\A_\phi$ ($\phi\in
\Phi(\A)$) is maximal. It is easy to show that the class of  maximal
subalgebras $\A_\phi$ coincides with the class $\set{V_c| \
c\in\Complex}$.
\end{example}
\begin{lemma}
For every $c\in \Complex$ and every $\phi\colon \A \to \B\subseteq \BOP(\h)$ such that
$\A_{\phi} = V_c$ we have $$ \norm{\phi} \norm{(\phi)^{-1}}
\ge \frac{1}{\abs{c}}.$$
\end{lemma}
\begin{proof}
Note that $\A$ as a linear space is $V_c\oplus \Complex \eta$ where
$\eta = \left(
          \begin{array}{cc}
            0 & 1 \\
            0 & 0 \\
          \end{array}
        \right)
$ and $V_c$ is a linear span of $I$ and $\xi = \left(
                                                 \begin{array}{cc}
                                                   0 & 1 \\
                                                   0 & c \\
                                                 \end{array}
                                               \right)
$. Denote $A = \phi(\xi)$, $X =  \phi(\eta)$. Since
$\A_{\phi} = V_c$, $\phi(V_c)$ is a $C\sp*$-algebra. Hence
$A^*\in \phi(V_c)$. From this follows that $A^*$ is a linear
combination of $I$ and $A$. Thus $A$ is a normal operator with
two-point spectrum $\set{0, c}$. It is also easy to check the
following relations $$ X^2 = 0, X A = c X, A X =0 . $$
Thus there are Hilbert spaces $\h_1$, $\h_2$ and  a unitary operator
$U:\h_1\oplus \h_2\to\h$ such that
$$ U^* A U = \left(
                  \begin{array}{cc}
                    0 & 0 \\
                    0 & c I \\
                  \end{array}
                \right), U^* X U = \left(
                  \begin{array}{cc}
                    0 & Y \\
                    0 & 0 \\
                  \end{array}
                \right).
$$ In the above matrices zero denotes zero operators between $\h_i$ and $\h_j$ and
$Y:\h_2\to\h_1$. We have  $$\phi(\left(
                                   \begin{array}{cc}
                                     \alpha & \beta \\
                                     0 & \gamma I\\
                                   \end{array}
                                 \right)
)= U \left(
         \begin{array}{cc}
           \alpha & (\beta - \frac{\gamma-\alpha}{c})Y \\
           0 & \gamma \\
         \end{array}
       \right)
 U^*.$$ Thus $$\norm{\phi} \ge \norm{\left(
                                           \begin{array}{cc}
                                             0 & \frac{1}{c}Y \\
                                             0 & I \\
                                           \end{array}
                                         \right)
}\ge \frac{1}{\abs{c}} \norm{Y}.$$ Taking $\alpha= \gamma= 0,
\beta=1$ we obtain $$\norm{\phi^{-1}}\ge \norm{\left(
                                           \begin{array}{cc}
                                             0 & Y \\
                                             0 & 0 \\
                                           \end{array}
                                         \right)
}^{-1}= \norm{Y}^{-1}.$$ Hence $\norm{\phi}\norm{\phi^{-1}}\ge
\frac{1}{\abs{c}}$.
\end{proof}
In the following example we show that the arguments used to prove
Theorem~\ref{ultra} could not be carried out in the case of
$\Phi$-maximal subalgebras.
\begin{example}
Let $\B= \oplus_{n=1}^\infty \A_n$ where $\A_n$ is a copy of $\A$
from Example~\ref{examp}. Let $\mathcal{D}$ be $\Complex I$.
Consider the increasing sequence of subalgebras
$$
\B_n= V_1\oplus V_{1/2} \oplus\ldots \oplus
V_{1/n}\oplus \Dmc\oplus\Dmc \ldots.
$$
Let $\phi_{1/n}$ be such that
$\B_{\phi_{1/n}} = V_{1/n}$ and $\B_{\phi_0} = \mathcal{D}$. Taking
$$
\psi_n = \phi_1\oplus \phi_{1/2}\oplus \ldots \oplus \phi_{1/n}
\oplus \phi_0 \oplus \phi_0 \ldots
$$  we see that $\psi_n \in
\Phi(\B)$ and $\B_{\psi_n} = \B_n$. But there is no $\tau\in
\Phi(\B)$ such that $\B_{\psi_n}\subseteq \B_\tau$ for all $n$.
Indeed otherwise denote $\tau_{1/n}$ the restriction of $\tau$ on
$\A_n$. Then $\tau_{1/n}\in \Phi(\A_n)$ and $(\A_n)_{\tau_{1/n}} =
V_{1/n}$. The letter equality follows from the fact  that $ V_{1/n}
\subseteq (\A_n)_{\tau_{1/n}}$ (since $\B_n\subseteq \B_\tau$) and
maximality of  $V_{1/n}$. Thus
$\norm{\tau_{1/n}}\norm{\tau_{1/n}^{-1}}\ge n$ and $\tau\not\in
\Phi(\B)$.
\end{example}

The property of $\Phi$-maximal subalgebras stated in the following
proposition is conjectured to be characteristic for $\Phi$-maximal
subalgebras.

\begin{proposition} Let $\B$ be an operator algebra.
Suppose that $\tau:\B\to \BOP(\widetilde{\h})$  is $\Phi$-maximal
(or $\Phi_r$-maximal) and a homomorphism
 $\rho:\B\to \BH$ lies in $\Phi(\B)$ (resp.
$\rho\in \Phi_r(\B)$) and that $\tau$ dilates $\rho$ (i.e.
$\h\subseteq \widetilde{\h}$  and $\rho(b) = P_{\h}\tau(b)|_\h$ for
all $b\in \B$). Then $\tau(\B_\rho)$ leaves $\h$ invariant.
\end{proposition}
\begin{proof}
If $b\in \B$ and $\tau(b)$ is self-adjoint then the equality
$\rho(b) = P_{\h}\tau(b)|_\h$ implies that $\rho(b)$ is also
self-adjoint. Thus $b\in\B_{\rho}$. The set of all such $b$
generates $\B_\tau$ which give an inclusion $\B_\tau\subseteq
\B_\rho$. Since $\tau $ is $\Phi$-maximal the converse inclusion
also holds. Thus $\B_\tau = \B_\rho$. By Sarason's theorem
(see~\cite{Sarason})  $\h$ is a semiinvariant subspace in
$\widetilde{\h}$. Hence there are $\tau(\B)$-invariant subspaces
$\h_1$ and $\h_2$ in $\widetilde{\h}$ such that $\h_1\oplus\h
=\h_2$. With respect to this decomposition into orthogonal direct
sum $\tau|_{\h_2}$ has the following matrix $\left(
   \begin{array}{cc}
     \tau|_{\h_1}(b) & D(b) \\
     0 & \rho(b) \\
   \end{array}
 \right)
$ where $D(b)$ is a map from $\h_2$ to $\h_1$. For $b\in B_\rho$
such that $\rho(b)$ is self-adjoint $\tau|_{\h_2}$ is also
self-adjoint. Hence $D(b)=0$. From this follows that $D(B_\rho)=0$
and consequently $\tau(\B_\rho)$ leaves $\h$ invariant.
\end{proof}

\section{Appendix}
For the convenience of the reader we review some relevant facts from
noncommutative Gr\"obner bases theory (see~\cite{Ufn}).

Let $X$ be an alphabet and  $\F(X)$ be the free associative algebra
over $X$. Denote by $W$ the set of all words in $X$ including the
empty word $e$. In particular $W$ is a linear basis for $\F(X)$. For
a word  $w= x_{1}^{\alpha_1} \ldots
 x_{k}^{\alpha_k}$ (where $x_j\in X$ and $\alpha_j\in \mathbb{N}$) the length of $w$,
 denoted by $l(w)$, is defined to be $\alpha_1+\ldots +\alpha_k$.

Assume that $W$ is given an admissible ordering, i.e. a
well-ordering such that $c<d$ implies $a c b < a d b$ for all $a, b,
c, d \in W$. It is customary to use the following admissible
ordering called {\it deglex} (degree lexicographic). Fix any
well-ordering on $X$. For $c$ and $d$ in $W$ we set $c>d$ if either
$l(c)>l(d)$ or $l(c) =l(d)$ and $c$ is larger than $d$
lexicographically.

   Any  $f \in \F(X)$  is a linear combination $\sum_{i=1}^k
\alpha_k w_i$  of distinct words $w_1$, $w_2$, $\ldots$, $w_k$ with
complex coefficients $\alpha_i\not=0$ for all $i\in \{1,\ldots, k
\}$.  Let $\hat{f}$ denote  the greatest (in the fixed admissible
order) of these words, say $w_j$. The coefficient $\alpha_j$ denoted
by ${\rm lc} (f)$ is called {\it leading coefficient}. Denote
$\hat{f}- {\rm lc} (f)^{-1}f$ by $\bar{f}$. The degree of $f\in
\F(X)$, denoted by $\deg(f)$, is defined to be $l(\hat{f})$.
Elements of the free algebra $\F(X)$ can be identified with
functions $f:W\to \mathbb{C}$ with finite support via the map $f\to
\sum_{w\in W}f(w)w$. For a word $z\in W$ and an element $f\in \F(X)$
we will write $z\prec f$ if $f(z)\not= 0$.
\begin{definition}
We will say that two elements $f,g \in \F(X)$ form a composition
$w\in W$ if there are words $x,z \in W$ and a non-empty word $y\in
W$ such that $\hat{f} = x y$, $ \hat{g} = y z$ and $w=x y z$. Denote
the result of the composition ${\rm lc} (g) f z-{\rm lc} (f) x g$ by
 $(f,g)_w$.
\end{definition}
If $f$ and $g$ are as in the preceding definition then  $f= {\rm lc}
(f) x y +{\rm lc} (f) \bar{f}$, $g= {\rm lc} (g) y z +{\rm lc} (g)
\bar{g}$ and $(f,g)_w= {\rm lc} (f) {\rm lc} (g) (\bar{f} z - x
\bar{g})$. We will also say that $f$ and $g$ intersect by $y$.

\begin{definition}
A subset $S \subseteq \F(X)$ is called closed under compositions  if
for any two elements  $f$, $g \in S$ the following properties holds.
\begin{enumerate}
\item If $f\not=g$ then the word $\hat{f}$ is not a subword in
$\hat{g}$. \item  If $f$ and $g$ form a composition $w$ then there
are words $a_j$, $b_j$ $\in W$,  elements
 $f_j \in S$  and complex   $\alpha_j$ such that $(f,g)_w=\sum_{j=1}^m \alpha_ja_jf_jb_j$ and  $a_jf_jb_j < w$, for $j=1,\ldots,m$.
\end{enumerate}
\end{definition}

\begin{definition}
A set $S\subseteq \F(X)$  is called a Gr\"obner basis of an  ideal
$\mathcal{I}\subseteq \F(X)$ if for any $f\in \mathcal{I}$ there is
$s\in S$ such that $\hat{s}$ is a subword in $\hat{f}$. A Gr\"obner
basis $S$ of $\mathcal{I}$ is called minimal if no proper subset of
$S$  is a Gr\"obner basis of $\mathcal{I}$.
\end{definition}

If $S$ is closed under compositions then $S$ is a minimal Gr\"obner
basis for the ideal $\mathcal{I}$ generated by $S$ (see \cite{Bok}).
Henceforth we will consider only minimal Gr\"obner bases.  Thus we
will say that $S$ is a Gr\"obner basis of an associative algebra
$\A=\F(X)/\mathcal{I}$ if $S$ is closed under composition and
generates $\mathcal{I}$ as an ideal of $\F(X)$. It could be proven
using Zorn's lemma that for any generating set $X$ of $\A$ and any
admissible ordering there exists a Gr\"obner basis of $\A$ in
$\F(X)$.

Let $S$ be a Gr\"obner basis for $A$ and let  $\hat{S}=\{ \hat{s} |
s\in S\}$. Denote by $BW(S)$ the subset of those words in $W$ that
contain no word from $\hat{S}$ as a subword. It is a well known fact
that $BW(S)$ is a linear basis for $\A$.

If $S\subseteq\F(X)$ is closed under compositions and $\mathcal{I}$
is an  ideal generated by $S$ then each element $f+\mathcal{I}$ of
the factor algebra $\F(X)/\mathcal{I}$ is the unique linear
combination of basis vectors $\{w +\mathcal{I}\}_{w \in BW}$  $$
f+\mathcal{I} = \sum_{i=1}^n c_i (w_i+ \mathcal{I}). $$

We can define an operator ${\rm R}_S:\F(X)\to \F(X)$ by the
following rule ${\rm R}_S(f) = \sum_{i=1}^n c_i w_i$. The element
${\rm R}_S(f)$ is called a canonical form of the element $f$ in the
factor algebra $\F(X)/\mathcal{I}$.

\end{document}